\documentclass{amsart} 
\usepackage{amssymb}
\usepackage{hyperref}
\usepackage{amsmath,amsthm}
\usepackage{amssymb,xcolor}
\usepackage{a4wide}
\hypersetup{linkcolor=blue, colorlinks=true
,citecolor = red}
\newtheorem{thm}{Theorem}[section]
\newtheorem{Def}{Definition}[section]
\newtheorem{Rem}{Remark}[section]

\newtheorem{corl}{Corollary}[section]

\newtheorem{lem}{Lemma}[section]
\date{}
\catcode`\@=11
\def\theequation{\@arabic{\c@section}.\@arabic{\c@equation}}
\catcode`\@=12

\newcommand{\Om} {\Omega}

\newcommand{\be} {\begin{equation}}
\newcommand{\ee} {\end{equation}}
\newcommand{\bea} {\begin{eqnarray}}
\newcommand{\eea} {\end{eqnarray}}
\newcommand{\Bea} {\begin{eqnarray*}}
\newcommand{\Eea} {\end{eqnarray*}}

\newcommand{\De} {\Delta}

\newcommand{\D}{(-\Delta)^m}

\numberwithin{equation}{section}
\newtheorem{theorem}{Theorem}[section]
\newtheorem{proposition}[theorem]{Proposition}

\allowdisplaybreaks

\begin{document}
\title[Multiplicity of a Critical Growth Inhomogeneous Equation]{Existence and Multiplicity of a Nonhomogeneous Polyharmonic Equation With Critical Exponential Growth in Even Dimension}
\author{Abhishek Sarkar}
\address{TIFR Centre For Applicable Mathematics\\
Post Bag No. 6503, Shardanagar, Bangalore-560065, Karnataka, India.}
\email{abhishek@math.tifrbng.res.in}
\begin{abstract}
In this paper we study the existence of at least two positive weak solutions for an inhomogeneous fourth order equation with Navier boundary data involving nonlinearities of critical growth with a bifurcation parameter $\lambda$ in $\mathbb{R}^{2m}$. We establish here the lower and upper bound for $\lambda$ which determine multiplicity and non-existence respectively. 
\end{abstract}

\maketitle

\section{Introduction}
Let $\Omega \subset \mathbb{R}^{2m}$ be a bounded domain. In this context we study the existence of multiple solutions in $W^{m,2}_{\mathcal{N}}(\Om) = \{u \in W^{m,2}(\Omega): \Delta^j u = 0 \mbox{ on } \partial \Omega \text{ for } 0 \leq j < \frac{m}{2}\}$ of the following $2m$-th order problem with Navier boundary condition
$$(P)\hspace{1cm}\left\{
\begin{array}{lllll}\left.
\begin{array}{lllll}
\D u &= \mu u |u|^pe^{u^2} + \lambda h(x)\\
\hspace{10mm} u &> 0 \end{array}\right\} \;\; \text{in} \;\Om, \\
u=\De u=0=..=\De^{m-1}u \hspace{15mm}\text{ on } \partial \Omega \end{array}\right.
$$
 where $h \geq 0$ in $\Omega$, $\|h\|_{L^2(\Omega)} = 1,$ $ \lambda >0, \mu =1$ if $p>0$ and $\mu \in (0,\lambda_1(\Omega))$ if $p=0$. Also assume that $\lambda_1(\Omega)$ be the first eigenvalue of $\D $ on $W^{m,2}_{\mathcal{N}}(\Omega)$ with Navier boundary condition and which is strictly positive. The existence of multiple solutions for analogous problems in higher dimension with critical exponent have been studied in \cite{Lu}, \cite{Peral} for the Dirichlet boundary condition and in \cite{Zhang} for Navier boundary condition. The corresponding problem for second order elliptic equations have been studied in \cite{Prashanth} for dimension two, and in \cite{Tarantello} for higher dimensions. The critical growth for the nonlinearity is $u \mapsto |u|^pu$, $p = \frac{4m}{n-2m}$, when $n \geq 2m+1$ 
 from the Sobolev embedding in $\mathbb{R}^n$. In \cite{Moser} Moser proved the following,  
 \begin{thm}          \label{Moser_Thm}
 Let $\Om \subset \mathbb{R}^n, n \geq 2$ be a bounded domain. There exists a constant $C_n >0$ such that for any $u \in W^{1,n}_0(\Omega)$, $n\geq 2$ with $\|\nabla u\|_{L^n(\Omega)} \leq 1$, then 
 \begin{equation}
  \int_{\Omega} e^{\alpha |u|^p}dx \leq C_n|\Omega|,              \label{Moser_Sharp}
 \end{equation}
where $$p = \frac{n}{n-1}, \hspace{1mm}\alpha_n:= n w^{\frac{1}{n-1}}_{n-1},$$ and $w_{n-1}$ is the surface measure of the unit sphere $\mathbb{S}^{n-1} \subset \mathbb{R}^n$. Furthermore the integral on the left hand side can be made arbitrarily large if $\alpha > \alpha_n$ by appropriate choice of $u$ with $\|\nabla u\|_{L^n(\Omega)} \leq 1$. The embedding $$W^{1,n}_0(\Omega) \ni u \mapsto e^{\alpha |u|^{\frac{n}{n-1}}} \in L^1(\Omega)$$ is compact for $\alpha < \alpha_n$ and is not compact for $\alpha = \alpha_n$.  
\end{thm}
In \cite{Adams} Adams extended the above result of Moser to higher order Sobolev spaces. To state the result of Adams we define the following $m$-th order derivatives of $u \in C^m(\Omega)$:
\[ \nabla^m u = \left\{
            \begin{array}{lllll}
             \Delta ^{\frac{m}{2}} u    \hspace{1cm} \text{for } m \text{ even},\\
             \nabla \Delta^{\frac{m-1}{2}} u \hspace{3mm} \text{for } m \text{ odd.}
            \end{array}\right.\]
            Furthermore, $\| \nabla^m u\|_p$ is the $L^p$ norm of the function $|\nabla^mu|,$ the usual Euclidean length of the vector $\nabla^m u$. We also denote $W^{m,\frac{n}{m}}_0(\Omega)$ to be the completion of $C^{\infty}_0(\Omega)$ under the Sobolev norm
            \begin{equation}
            \|u\|_{W^{m,\frac{n}{m}}(\Omega)}= \left( \|u\|^{\frac{n}{m}}_{\frac{n}{m}} + \sum_{|\alpha|=1}^m \|D^{\alpha} u\|^{\frac{n}{m}}_{\frac{n}{m}}\right)^{\frac{m}{n}}.              \label{full-norm}
            \end{equation}
 \par Adams proved the following embedding:
 \begin{thm}\label{Adams_Thm}
  Let $\Omega \subset \mathbb{R}^n$ be a bounded domain. If $m$ is a positive integer less than $n$, then there exists a constant $C_0=C(n,m)>0$ such that for any $u \in W^{m,\frac{n}{m}}_0(\Omega)$ with $\| \nabla^m u\|_{\frac{n}{m}} \leq 1,$ then 
 \begin{equation}
  \frac{1}{|\Omega|} \int_{\Omega} \exp(\beta |u(x)|^{\frac{n}{n-m}})dx \leq C_0,              \label{Adams_Sharp}
  \end{equation}
   for all $\beta \leq \beta_{n,m}$ where 
 \[\beta_{n,m}=\left\{
\begin{array}{lllll}
\frac{n}{w_{n-1}}\left[\frac{\pi ^{\frac{n}{2}}2^m \Gamma(\frac{m+1}{2})}{\Gamma(\frac{n-m+1}{2})}\right]^{\frac{n}{n-m}} \hspace{3mm} \text{when } m \text{ is odd,}\\
\frac{n}{w_{n-1}}\left[\frac{\pi ^{\frac{n}{2}}2^m \Gamma(\frac{m}{2})}{\Gamma(\frac{n-m}{2})}\right]^{\frac{n}{n-m}} \hspace{3mm} \text{when } m \text{ is even.}
\end{array}\right.\]
Furthermore, for any $\beta > \beta_{n,m}$, the integral can be made as large as possible by appropriate choice of $u$ with $\|\nabla ^m u\|_{\frac{n}{m}} \leq 1$. 
\end{thm}
\begin{Rem}
We remark that for the case $n=2m=4$, Lu-Yang in \cite{Lu-Yang} and in general Zhao-Chang \cite{Zhao-Chang} showed the existence of an explicit sequence for $n=2m$ to prove the sharpness of the constant in $W^{m,\frac{n}{m}}_0(\Omega)$.
\end{Rem}

Let $W^{m,\frac{n}{m}}_{\mathcal{N}}(\Omega)$ denote the following subspace of $W^{m,\frac{n}{m}}(\Omega)$:
$$W^{m,\frac{n}{m}}_{\mathcal{N}}(\Omega) = \left\{u \in W^{m,\frac{n}{m}}(\Omega): \Delta^j u =0, \text{ on } \partial \Omega \text{ for } 0 \leq j\leq [ (m-1)/2]\right\}.$$ Note that $W_0^{m,\frac{n}{m}}(\Omega)$ is strictly contained in $W^{m,\frac{n}{m}}_{\mathcal{N}}(\Omega)$.
Therefore,
 \[ \sup_{u \in W^{m,\frac{n}{m}}_0(\Omega), \|\nabla^m u\|_{\frac{n}{m}} \leq 1} \int_{\Omega} exp(\beta_{n,m} |u|^{\frac{n}{n-m}})dx \\ 
 \leq\\ \sup_{u \in W^{m,\frac{n}{m}}_{\mathcal{N}}(\Omega), \|\nabla^m u\|_{\frac{n}{m}} \leq 1} \int_{\Omega} exp(\beta_{n,m} |u|^{\frac{n}{n-m}})dx. \] 
 
 Tarsi \cite{Tarsi} later extended Adams' result for the larger space $W^{m,\frac{n}{m}}_{\mathcal{N}}(\Om)$. The key step in her work is to embed $W^{m,\frac{n}{m}}_{\mathcal{N}}(\Om)$ into a Zygmund space. We state her embedding theorem below
 \begin{thm}\label{Tarsi_Thm}
 Let $n >2$ and $\Omega \subset \mathbb{R}^n$ be a bounded domain. Then there is a constant $C_n >0$ such that for all $u \in W^{m,\frac{n}{m}}_{\mathcal{N}}(\Omega)$ with $\|\nabla^m u \|_{\frac{n}{m}} \leq 1$, we have
 \begin{equation}
  \int_{\Om} e^{\beta |u|^{\frac{n}{n-m}}}dx <C_n |\Omega|  \hspace{1cm} \forall \beta \leq \beta_{n,m}               \label{Tarsi}
 \end{equation}
 and the constant $\beta_{n,m}$ appearing in (\ref{Tarsi}) is sharp and $\beta_{n,m}$ is same as in Theorem \ref{Adams_Thm}
 \end{thm}
  \begin{Rem}
  Here we remark that the bilinear form 
 \begin{align}
  (u,v) \mapsto \int_{\Omega}\nabla^m u \cdot \nabla^m v= \left\{\begin{array}{lllll}
           \displaystyle\int_{\Omega} \Delta^k u \Delta^k v     \hspace{19mm}   \text{if } m =2k,\\
              \displaystyle \int_{\Omega} \nabla(\Delta^k u) \cdot \nabla (\Delta^k v)  \hspace{3mm} \text{ if } m=2k+1,        \label{poly-norm}
                       \end{array}\right. 
 \end{align}
 defines a scalar product on $W^{m,2}_0(\Omega)$ and $W^{m,2}_{\mathcal{N}}(\Omega)$. Furthermore if $\Omega$ is bounded this scalar product induces a norm equivalent to (\ref{full-norm}).
 \end{Rem}
Therefore the above results imply that the problem $(P)$ nonlinearity of critical growth.
\begin{theorem}  \label{Multiplicity}
There exist positive real numbers  $\lambda_* \leq \lambda^*$ with $\lambda_*$ independent of $h$ such that the problem $(P)$ has at least two positive solutions for all $ \lambda \in (0, \lambda_*)$ and no solution for all $\lambda > \lambda^*$.
\end{theorem}
In spite of possible failure of Palais-Smale condition due to the presence of critical exponent we adapt the method of \cite{Tarantello} to prove the existence of the first solution by a decomposition of Nehari manifold into three parts. However for the existence of second solution we rely on the refined version of the Mountain-Pass Lemma introduced by Ghoussoub-Preiss \cite{Preiss}.
\section{Decomposition of Nehari Manifold}
Let $f(u) = \mu |u|^pue^{u^2}.$ The corresponding energy functional to the problem $(P)$ is given by 
\begin{equation}
J(u)= \frac{1}{2} \int_{\Omega}|\nabla^m u|^2- \int_{\Omega} F(u) - \lambda \int_{\Om} hu         \label{energy}
\end{equation}
where $F(u) = \int_0^u f(s)ds.$ As the energy functional is not bounded below on $W^{m,2}_{\mathcal{N}}(\Om)$, we need to study $J(u)$ on the Nehari manifold 
   \begin{equation}
   \mathcal{M}= \{u \in W^{m,2}_{\mathcal{N}}(\Om) \setminus \{0\}: \langle J'(u),u \rangle =0\},
   \end{equation}
where $J'(u)$ denotes the Frechet derivative of $J$ at $u$, and $\langle .,. \rangle$ is the inner product. Here we note that $\mathcal{M}$ contains every nonzero solution of the problem $(P)$.
We note that for any $u \in W^{m,2}_{\mathcal{N}}(\Om),$
\begin{align*}
\langle J'(u),u \rangle &=  \int_{\Omega} |\nabla^m u|^2 - \int_{\Om} f(u)u - \lambda \int_{\Om} hu, \\
\langle J''(u)u,u \rangle &= \int_{\Omega}|\nabla^m u|^2 - \int_{\Om} f'(u)u^2.
\end{align*}
Similarly to the method used in \cite{Tarantello}, We split $\mathcal{M}$ into three parts:
\begin{align*}
\mathcal{M}^0 = \{u \in \mathcal{M}: \langle J''(u)u,u \rangle = 0\},\\
\mathcal{M}^+ = \{u \in \mathcal{M}: \langle J''(u)u,u \rangle > 0\},\\
\mathcal{M}^- = \{u \in \mathcal{M}: \langle J''(u)u,u \rangle < 0\}.
\end{align*}
\section{Topological Properties of \texorpdfstring{$\mathcal{M}^0,\mathcal{M}^+, \mathcal{M}^-$}{Lg}}
Our first aim is to show, for some small $\lambda, \mathcal{M}^0 = \{0\}$. For this let $\zeta >0$, if $p>0$ and $\zeta < \frac{\lambda_1-\mu}{\mu}$ if $p=0$. Let $ \Lambda = \{u \in W^{m,2}_{\mathcal{N}}(\Om): \int_{\Om} |\nabla^m u|^2 \leq (1+\zeta)\int_{\Omega} f'(u)u^2\}.$
Lemma \ref{nonempty} implies that $\Lambda \neq \{0\}$. We now assume the following important hypothesis: 
\begin{align}
 \lambda >0, \|h\|_{L^2(\Om)} = 1, \text{ and}     \notag\\    
\inf_{u \in \Lambda \setminus \{0\}}\left(\mu \int_{\Om}(p+2u^2)|u|^{p+2}e^{u^2} - \lambda \int_{\Om}hu\right)>0.   \label{Nehari1}
\end{align}
The condition (\ref{Nehari1}) forces $\lambda$ to be suitably small. Indeed we can prove the following. 
\begin{proposition}      \label{lambdastar}
Let 
\begin{equation}
\lambda < \mu C_0^{\frac{p+3}{p+4}}|\Om|^{-\left(\frac{p+2}{2p+8}\right)} \label{lambda}   
\end{equation}
 where $C_0 = \inf_{u \in \Lambda \setminus \{0\}} \int_{\Om}(p+2u^2)|u|^{p+2}e^{u^2} >0.$ Then (\ref{Nehari1}) holds.
\end{proposition}
\begin{proof}
\textbf{Step 1:} $\inf_{ u \in \Lambda \setminus \{0\}} \|u\|_{W^{m,2}_{\mathcal{N}}(\Om)} >0.$ \\ 
Assume the contradiction, then there exists a sequence $\{u_n\} \subset \Lambda \setminus \{0\}$ such that $\|u_n\|_{W^{m,2}_{\mathcal{N}}(\Om)} \to 0$ as $n \to \infty$. Let $v_n = \frac{u_n}{\|u_n\|_{W^{m,2}_{\mathcal{N}}(\Om)}}.$ Then $\|v_n\|_{W^{m,2}_{\mathcal{N}}(\Om)} =1$ and $v_n$ satisfies 
\begin{equation}
1 \leq (1+\zeta) \int_{\Om} f'(u_n) v_n^2, \quad \forall n.      \label{lambda1}
\end{equation}
Since $u_n \to 0$ in $W^{m,2}_{\mathcal{N}}(\Om)$, by Adams' embedding for the higher order derivative in Theorem \ref{Tarsi_Thm} we get $f'(u_n) \to f'(0)$ in $L^r(\Om)$ for all $r \geq 1$. Since $v_n$ is bounded in $W^{m,2}_{\mathcal{N}}(\Om)$, $v_n$ has a weak limit say $v$ in $W^{m,2}_{\mathcal{N}}(\Om)$. Certainly $\|v\|_{W^{m,2}_{\mathcal{N}}(\Om)} \leq 1$ and up to a subsequence denote it same as $v_n$ which converges strongly to $v$ in $L^r(\Omega)$ for all $r \geq 1$. Hence from (\ref{lambda1}) we get 
\begin{equation}
\int_{\Omega} |\nabla^m u|^2 \leq 1 \leq (1+\zeta) f'(0) \int_{\Om} v^2.        \label{lambda2}
\end{equation} 
This gives a contradiction if $p>0$ in which case $f'(0)= 0$. If $p=0$, by assumption
$$\int_{\Omega} |\nabla^m u|^2 \geq \lambda_1 \int_{\Om} v^2 > (1+\zeta) \mu \int_{\Om}v^2$$ which gives a contradiction to (\ref{lambda1}) since $f'(0)= \mu$. This proves Step 1.\\ 
It is easy to check that using Step 1 and the definition of $\Lambda$:
\begin{equation}
 \inf_{u \in \Lambda \setminus \{0\}}\int_{\Omega}(p+2u^2)|u|^{p+2}e^{u^2} =C_0 >0.                   \label{lambda-inf}
\end{equation}

\noindent\textbf{Step 2:} Finally we have,
\begin{align*}
 \lambda \left|\int_{\Omega} hu\right| &\leq \lambda \|u\|_{L^2(\Omega)} \leq \lambda |\Omega|^{\frac{p+2}{2p+8}}\left(\int_{\Omega} |u|^{p+4} \right)^{\frac{1}{p+4}}\\
 &\leq \frac{\lambda |\Omega|^{\frac{p+2}{2p+8}}}{\mu (p+2u^2)|u|^{p+2}e^{u^2})^{\frac{p+3}{p+4}}} (\mu \int_{\Omega} (p+2u^2)|u|^{p+2}e^{u^2})\\
 &\leq \left(\frac{\lambda |\Omega|^{\frac{p+2}{2p+8}}}{\mu C_0^{\frac{p+3}{p+4}}}\right) (\mu \int_{\Omega}(p+2u^2)|u|^{p+2}e^{u^2}).
\end{align*}
Hence from the above inequality together with (\ref{lambda}) and (\ref{lambda-inf}) the proof is complete.
\end{proof}
\begin{lem}   \label{Null Set}
Suppose $\lambda >0$ be such that (\ref{Nehari1}) holds. Then $\mathcal{M}^0 = \{0\}$.
\end{lem}
\begin{proof}
Let $u \in \mathcal{M}^0, u \neq 0$. Then we have 
\begin{align}
\int_{\Omega}|\nabla^m u|^2 &= \int_{\Om} f(u) u + \lambda \int_{\Om}hu,           \label{Nehari2}\\
\int_{\Omega}|\nabla^m u|^2 &= \int_{\Om} f'(u)u^2.                          \label{Nehari3}
\end{align}
We note that from (\ref{Nehari3}) $$\int_{\Omega}|\nabla^m u|^2 = \int_{\Om} f'(u)u^2 < (1+\zeta) \int_{\Om} f'(u)u^2$$ it implies that $ u \in \Lambda \setminus \{0\}$. From these two expressions we get
$$\lambda \int_{\Om} hu = \int_{\Om}(f'(u)u-f(u))u = \mu \int_{\Om}(p+2u^2)|u|^{p+2}e^{u^2}$$
which violates the condition (\ref{Nehari1}). Therefore $\mathcal{M}^0 = \{0\}$.
\end{proof}

Next we are going to discuss the topological properties of $\mathcal{M}^+$ and $\mathcal{M}^-$. Given $u \in {W^{m,2}_{\mathcal{N}}(\Om)} \setminus \{0\}$, we define $\xi_u: \mathbb{R}^+ \to \mathbb{R}$ by 
\begin{equation}
\xi_u(s) = s \int_{\Om} |\nabla^m u|^2 - \int_{\Om} f(su)u.      \label{xi}
\end{equation} 
The choice of the above function is consequence of the following expression,
\begin{equation*}
\langle J'(su), su \rangle = s \left(s \int_{\Om} |\nabla^m u|^2 - \int_{\Om} f(su)u - \lambda \int_{\Om}hu \right).      
\end{equation*}
So, $\xi_u(s)=\lambda \int_{\Om}hu$ if and only if $su \in \mathcal{M}$, for $s >0$.

Now we are ready to state the following lemma.
\begin{lem}      \label{nonempty}
For every $u \in W^{m,2}_{\mathcal{N}}(\Om) \setminus \{0\}$ there exists a unique $s_* =s_*(u) >0$ such that $\xi_u(.)$ has its maximum at $s_*$ with $\xi_u(s_*)>0$. Also there holds $s_*u \in \Lambda \setminus \{0\}$.
\end{lem}
\begin{proof}
Differentiating (\ref{xi}) we have, 
\begin{equation}
\xi'_u(s) = \int_{\Om} |\nabla^m u|^2- \int_{\Om} f'(su)u^2.     \label{xi1}
\end{equation}
Observe that, 
\begin{align}
s^2 \xi'_u(s)&=  \int_{\Om}|\nabla^m (su)|^2 - \int_{\Om}f'(su)(su)^2     \notag \\
             &= \langle J''(su)su,su \rangle.                  \label{xi2}
\end{align}
Now we note that, $\xi_u(.)$ is strictly concave function on $\mathbb{R}^+$, since 
\begin{equation}
\xi''_u(s)=-\int_{\Om}f''(su)u^3 <0.               \label{xi3}
\end{equation}
Also from the range of $\mu$ we get 
\begin{align*}
\lim_{s \to 0+} \xi'_u(s) >0 \text{ and}\\
\lim_{s \to \infty} \xi_u(s) = -\infty.
\end{align*}
Hence there exists a unique maximum point of $\xi_u(.),$ say $s_* = s_*(u)>0$. Now using (\ref{xi1}) at $s =s_*$ in the definition of $\xi_u$, we deduce,
 \begin{align}
 \xi_u(s_*)&= s_* \int_{\Om} f'(s_*u) u^2 - \int_{\Om} f(s_*u)u,\quad \text{since } \xi'_u(s_*)=0  \notag\\
           &= \frac{1}{s_*} \int_{\Om} (f'(s_*u)s_*u - f(s_*u))s_*u  \notag\\
           &=\frac{\mu}{s_*} \int_{\Om} (p+2(s_*u)^2)|s_*u|^{p+2}e^{(s_*u)^2} > 0.     \label{xi4}
 \end{align}
 here we note that $f'(s)s-f(s) =\mu(p+2s^2)|s|^p s e^{s^2}$. Finally $$s_* \xi'_u(s_*) = \int_{\Om} |\nabla^m(s_*u)|^2 - \int_{\Om}f'(s_*u)(s_*u)^2 =0$$ which implies that $s_*u \in \Lambda \setminus \{0\}$.
\end{proof}

\begin{lem}
Let $\lambda$ be such that (\ref{Nehari1}) holds. Then, for every $u \in W^{m,2}_{\mathcal{N}}(\Om) \setminus \{0\}$, there exists a unique $s_ = s_-(u)>0$ such that $s_-u \in \mathcal{M}^-, s_- > s_*$ and $J(s_-u) = \max_{s \geq s_*} J(su) \hspace{2mm}\forall s \in [s_*, \infty), s \neq s_-.$ Also if we assume $\int_{\Om} hu >0$, then there exists a unique $s_+=s_+(u) >0$ such that $s_+u \in \mathcal{M}^+$. In particular $ s_+ <s_*$ and $J(s_+u) \leq J(su)$ for all $s \in [0, s_-].$      \label{maxima-minima}
\end{lem}
\begin{proof}
Define the functional $\rho_u :[0, \infty) \to \mathbb{R}$ by $\rho_u(s) = J(su)$. Then it is easy to verify that $\rho_u \in C^2([0,\infty],\mathbb{R}) \cap C((0,\infty),\mathbb{R})$. Then we have $$\rho'_u(s) = \xi_u(s)-\lambda \int_{\Om} hu, \hspace{2mm}\rho''_u(s)= \xi'_u(s), \hspace{1mm} \forall t >0.$$
Now from (\ref{Nehari1}) and (\ref{xi4}) we have, 
\begin{equation*}
\xi_u(s_*) - \lambda \int_{\Om} hu = \frac{1}{s_*}\left\{ \mu \int_{\Om} (p+2(s_*u)^2)|s_*u|^{p+2}e^{(s_*u)^2} - \lambda \int _{\Om} h(s_*u)\right\} >0
\end{equation*}
Since $\xi_u(.)$ is strictly decreasing in $(s_*,\infty)$ and $\lim_{t \to \infty} \xi_u(s) = -\infty$, there exists a unique $s_- = s_-(u) > s_*$ such that $\xi_u(s_-) = \lambda \int_{\Om} hu.$ That is $s_-u \in \mathcal{M}$. One has $s_+ > s_*$ and $\rho_u'(s_) <0$, we get $s_-u \in \mathcal{M^-}$.

On the other hand when $\int_{\Om} hu >0$ we have $\lim_{s \to 0+} \xi_u(s) <0$ and which gives for $s$ close to $0$, $\xi_u(s)- \lambda \int_{\Om} hu <0$. Hence there exists a unique $s_+$ such that $\xi_u(s_+) =  \lambda \int_{\Om} hu$ which implies $s_+u \in \mathcal{M}.$ From the graph we see that $\xi_u(.)$ is strictly decreasing in $(0,s_*)$. Hence we have $s_+u \in \mathcal{M}^+$. 

And the remaining properties of $s_-, s_+$ can be proved by analyzing the identity $\rho_u(s) = \xi_u(s) - \lambda \int_{\Om}hu.$
\end{proof}
\begin{Rem}
If we define the positive cone $\mathcal{P}=\{ u \in W^{m,2}_{\mathcal{N}}(\Om): \int_{\Om}hu >0\}$ in $W^{m,2}_{\mathcal{N}}(\Om)$. Then we note that $\mathcal{M^+} \subset \mathcal{P}$.
\end{Rem}
The next corollary shows some topological properties of $\mathcal{M}^+, \mathcal{M}^-$.
\begin{corl}
Let $S_{W^{m,2}_{\mathcal{N}}(\Om)} = \{u \in W^{m,2}_{\mathcal{N}}(\Om): \|u\|_{W^{m,2}_{\mathcal{N}}(\Om)} = 1\}$. Then there exists a diffeomorphism  $S^+: S_{W^{m,2}_{\mathcal{N}}(\Om)} \to \mathcal{M}^-$ defined by $S^+(u) = s_+(u)u$. Also $\mathcal{M}^+$ is homeomorphic to $S_{W^{m,2}_{\mathcal{N}}(\Om)} \cap \mathcal{P}$.
\end{corl}
\begin{proof}
The function $S^+$ is continuous because $s_+$ is continuous as an application of implicit function theorem applied to $(s,u) \to \xi_u(s) - \lambda \int_{\Om}hu$. And we deduce the continuity of $(S^+)^{-1}$ by the fact that $(S^+)^{-1}(w) = \frac{w}{\|w\|}.$ In a similar manner we can prove that $ \mathcal{M}^+$ is homeomorphic to $S_{W^{m,2}_{\mathcal{N}}(\Om)} \cap \mathcal{P}$.
\end{proof}

Relying on the embedding of $W^{m,2}_{\mathcal{N}}(\Om) \hookrightarrow L^q(\Om)$ for all $1 \leq q <\infty$ and using the estimate $F(s) \leq \frac{\mu|s|^p}{2}(e^{s^2}-1),$ for all $s \in \mathbb{R}$ we have the following lemma on the lower bound and upper bound.
\begin{lem}   \label{upper-lower}
There exists $C_1, C_2>0$ such that $$ -C_2 \lambda^{2p+8} \geq \theta_0 \geq -C_1 \lambda^{\frac{p+3}{p+4}}.$$ Where, $\theta_0 = \inf\{J(u) : u \in \mathcal{M}\}.$ 
\end{lem}
\begin{proof}
We prove the case of the lower bound.\\
Let $u \in \mathcal{M}$ then from the definition,
\begin{align*}
 J(u)=\frac{1}{2} \int_{\Omega} |\nabla^m u|^2 -\int_{\Omega}F(u) -\lambda \int_{\Omega}hu\\
    =\int_{\Omega} \left[ \frac{1}{2}f(u)u-F(u)\right] -\frac{\lambda}{2} \int_{\Omega}hu.
\end{align*}
We note that a simple calculation gives 
\begin{equation}
 F(t) \leq \frac{\mu |t|^p}{2}(e^{t^2}-1), \quad \mbox{ for all } s \in \mathbb{R}.                  \label{F-estimate}
\end{equation}
Using (\ref{F-estimate}) we get 
\begin{align}
 J(u) &\geq \frac{\mu}{2} \int_{\Omega} \left((u^2-1)e^{u^2}+1\right) - \frac{\lambda}{2} \int_{\Omega}hu    \notag \\
      &\geq \frac{c\mu}{2} \int_{\Omega} |u|^{p+4} -\frac{\lambda}{2}\int_{\Omega} hu,        \label{lower-bound-lemma}
\end{align}
since $(s^2-1)e^{s^2}+1 \geq cs^4$ for some $c>0$, for all $s \in \mathbb{R}$. By an application of Holder inequality we get 
\begin{equation}
 \int_{\Omega} hu \leq |\Omega|^{\frac{p+2}{2(p+4)}} \|u\|_{L^{p+4}(\Omega)}.          \label{lower-bound-holder}
\end{equation}
From (\ref{lower-bound-lemma}) and (\ref{lower-bound-holder}) we get,
\begin{equation}
 J(u) \geq \frac{c \mu}{2} \|u\|^{p+4}_{L^{p+4}(\Omega)} - \left( \frac{\lambda |\Omega|^{\frac{p+2}{2(p+4)}}}{2}\right)\|u\|_{L^{p+4}(\Omega)}.   \label{lower-bound}
\end{equation}
By considering the global minimum of the function 
\[ \omega(x)= \left(\frac{c \mu}{2}\right) x^{p+4} - \left( \frac{\lambda |\Omega|^{\frac{p+2}{2(p+4)}}}{2}\right) x,\] 
It can be shown that \[J(u) \geq -C_1 \lambda^{\frac{p+4}{p+3}}.\]
      In a similar fashion we can prove the upper bound for $J$.
\end{proof}

As a consequence of Lemma \ref{Null Set} we have:
\begin{lem} \label{Implicit-Inner}
Let $\lambda$ and $h$ satisfy (\ref{Nehari1}). Given $u \in \mathcal{M} \setminus \{0\}$ there exists $\delta >0$ and a differentiable function $s:\{w \in W^{m,2}_{\mathcal{N}}(\Om): \|w\|_{W^{m,2}_{\mathcal{N}}(\Om)} < \delta \}\to \mathbb{R}$, with $$s(0)=1, s(w)(u-w) \in \mathcal{M}, \hspace{1.5mm}\forall \quad \|w\|_{W^{m,2}_{\mathcal{N}}(\Om)} < \delta$$ and 
\begin{equation}
\langle s'(0), v \rangle= \frac{2 \int_{\Om} \nabla^m u \cdot \nabla^m v - \int_{\Om} (f'(u)u+f(u))v - \lambda \int_{\Om}hv}{\int_{\Om} |\nabla^m u |^2 - \int_{\Om}f'(u)u^2}                        \label{Nehari4}
\end{equation}                           
\end{lem}

\begin{proof}
We define the function  $G: \mathbb{R} \times W^{m,2}_{\mathcal{N}}(\Om) \to \mathbb{R}$ by, 
$$G(s,w) = s \int_{\Om} |\nabla^m (u-w)|^2 - \int_{\Om} f(s(u-w))(u-w) - \lambda \int_{\Om} h(u-w).$$
Then $ G \in C^1( \mathbb{R} \times W^{m,2}_{\mathcal{N}}(\Om); \mathbb{R})$ and since $u \in \mathcal{M}$ it implies $G(1,0)= \int_{\Om} |\nabla^m u|^2 - \int_{\Om}f(u)u - \lambda \int_{\Om} hu = 0.$ Also $G_s(1,0) \neq 0$, indeed $G_s(1,0)= \int_{\Om} |\nabla^m u|^2 - \int_{\Om} f'(u)u^2 \neq 0$ thanks to Lemma \ref{Null Set}. Then by the Implicit Function Theorem, there exists $\delta >0, s:\{w \in W^{m,2}_{\mathcal{N}}(\Om): \|w\| < \delta\} \to \mathbb{R}$ of class $C^1$ that satisfies: 
\begin{align*}
G(s(w),w) &= 0 \text{ for all } w \in W^{m,2}_{\mathcal{N}}(\Om), \|w\|_{W^{m,2}_{\mathcal{N}}(\Om)} < \delta,\\
s(0) &= 1.
\end{align*}
Also 
\begin{align*}
0 &= s(w)G(s(w),w)\\
  &= \int_{\Om} (s(w) |\nabla^m (u-w)|)^2 - \int_{\Om}f(s(w)(u-w)(s(w)(u-w)) -\lambda \int_{\Om}h(s(w)(u-w)),
\end{align*} 
that is $s(w)(u-w) \in \mathcal{M}$ for all $w \in W^{m,2}_{\mathcal{N}}(\Om)$ with $\|w\| < \delta$.
Now if we differentiate the identity $G(s(w),w)=0$ with respect to $w$, we get 
\[ 0 = \langle G_s(s(w),w) + G_w(s(w),w), v \rangle \text{ for all } v \in W^{m,2}_{\mathcal{N}}(\Om). \] 
Putting $w=0$ in the above identity 
\[ 0 = \langle G_s(1,0) s'(0) + G_w(1,0), v \rangle = G_s(1,0) \langle s'(0), v \rangle + \langle G_w(1,0), v \rangle \] and we deduce from above 
\begin{align*}
\langle s'(0), v \rangle &= -\frac{\langle G_w(1,0),v \rangle}{G_s(1,0)}\\
                         &=  \frac{2 \int_{\Om} \nabla^m u \cdot \nabla^m v - \int_{\Om} (f'(u)u+f(u))v - \lambda \int_{\Om}hv}{\int_{\Om} |\nabla^m u|^2 - \int_{\Om}f'(u)u^2}.
  \end{align*}
\end{proof}

\section{Local Minimum of \texorpdfstring{$J$}  \text{ in} \texorpdfstring{$W^{m,2}_{\mathcal{N}}(\Om)$}.}
We are now in a situation to prove the existence of a minimizer for $J$ and hence we guarantee the existence of first solution. 
\par Since $\mathcal{M}$ is a closed set of $W^{m,2}_{\mathcal{N}}(\Om)$, hence a complete metric space. Now $J$ is bounded below on $\mathcal{M}$. By the Ekeland's Variational Principle there exists a sequence $\{u_n\} \subset \mathcal{M} \setminus \{0\}$ satisfying:
\begin{equation}
J(u_n) < \theta_0 + \frac{1}{n}, \hspace{2mm} J(v) \geq J(u_n)- \frac{1}{n}\|v-u_n\|_{W^{m,2}_{\mathcal{N}}(\Om)} \hspace{1mm} \forall v \in \mathcal{M}                  \label{Ekeland}
\end{equation}
\begin{proposition}
Let $\lambda$ and $h$ satisfy (\ref{Nehari1}). Then $$\lim _{n \to \infty}\|J'(u_n)\|_{(W^{m,2}_{\mathcal{N}}(\Om))^{-1}} = 0.$$ 
\end{proposition}
\begin{proof}
We proceed in a few steps. With the help of Lemma \ref{upper-lower} we've $\varliminf_{n \to \infty}\|u_n\|_{W^{m,2}_{\mathcal{N}}}>0$.
\textbf{Claim 1:} $\varliminf_{n \to \infty} \int_{\Omega} (p+2u_n^2)|u_n|^{p+2}e{u_n^2} >0$.\\
If possible let's assume that for a subsequence of $\{u_n\}$, which is still denoted by $\{u_n\}$, we have 
\begin{equation}
 \lim_{n \to \infty} (p+2u_n^2)|u_n|^{p+2}e^{u_n^2} \to 0 \text{    as   } n \to \infty      \label{lim-dual}
\end{equation}
Here we note that $u_n \to 0$ in $L^q(\Omega)$ for all $q \in [1,\infty)$ using (\ref{lim-dual}), and if $p>0$, 
\[\int_{\Omega} f(u_n)u_n = \mu \int_{\Omega} |u_n|^{p+2} e^{u_n^2} \to 0 \text{  as } n \to \infty.\] 
Therefore we have $\int_{\Omega}f(u_n)u_n \to 0,$  $\int_{\Omega} hu_n  \to 0$ as $n \to \infty$. Which imply that $\|u_n\|_{W^{m,2}_{\mathcal{N}}} \to 0$ as $n \to \infty$ because $\{u_n\} \subset \mathcal{M}$ hence a contradiction to the fact that $\varliminf_{n \to \infty}\|u_n\|_{W^{m,2}_{\mathcal{N}}}>0$. Similar argument for $p=0$.\\
\textbf{Claim 2:} $\varliminf_{n \to \infty} \{ |\int_{\Omega}|\nabla^m u|^2 - \int_{\Omega} f'(u_n)u_n^2| >0\}.$\\
Let the claim doesn't hold. Then for a subsequence $\{u_n\}$ we have  
\[ \int_{\Omega} |\nabla^m u|^2 - \int_{\Omega} f'(u_n)u_n^2 = o_n(1).\] 
From this and the fact $\varliminf_{n \to \infty}\|u_n\|_{W^{m,2}_{\mathcal{N}}}>0$ we deduce that,
\[\varliminf_{n \to \infty}f'(u_n)u_n^2 >0.\]
Therefore we have $u_n \in \Lambda \setminus \{0\}$ for large $n$. Since $\{u_n\} \subset \mathcal{M}$ we get
\begin{align*}
 o_n(1) &= \lambda \int_{\Omega} hu_n + \int_{\Omega} [f(u_n)-f'(u_n)u_n]u_n\\
          &= -\mu \int_{\Omega} (p+2u_n^2)|u_n|^{p+2}e^{u_n^2} + \lambda \int_{\Omega} hu_n,
\end{align*}
which contradicts (\ref{Nehari1}). This completes the proof of the claim.\\
Now we proof the theorem. Let's assume $\|J'(u_n)\|_{(W^{m,2}_{\mathcal{N}})^{-1}}  >0$ for all large $n$ (otherwise obvious). Now we define $u=u_n \in \mathcal{M}$ and $w=\delta \frac{J'(u_n)}{\|J'(u_n)\|}$ (by Riesz representation theorem, we identify $J'(u_n)$ as an element in $W^{m,2}_{\mathcal{N}}(\Omega)$ still denote $J'(u_n)$) for $\delta >0$ small. Therefore we can apply Lemma \ref{Implicit-Inner} for $w$ small we get $s_n(\delta):= s\left[ \delta \frac{J'(u_n)}{\|J'(u_n)\|}\right] >0$ such that \[ w_{\delta}= s_n(\delta) \left[ u_n - \delta \frac{J'(u_n)}{\|J'(u_n)\|}\right] \in \mathcal{M}.\]
Now from (\ref{Ekeland}) and a Taylor expansion we have:
\begin{align*}
 \frac{1}{n} \|w_{\delta} - u_n\| &\geq J(u_n) - J(w_{\delta})\\
                       &=(1-s_n(\delta)) \langle J'(w_{\delta}),u_n \rangle + \delta s_n(\delta) \left\langle J'(w_{\delta}), \frac{J'(u_n)}{\|J'(u_n)\|} \right\rangle +o(\delta)
                       \end{align*}
Dividing by $\delta>0$ and taking limit as $\delta \to 0$ we get:
\[\frac{1}{n}(1+|s_n'(0)| \|u_n\|) \leq -s_n(0) \langle J'(u_n),u_n\rangle +\|J'(u_n)\|=\|J'(u_n)\|.\] 
Hence \[ \|J'(u_n)\| \leq \frac{1}{n}(1+s_n'(0)|\|u_n\|).\]
We complete the proof by using, $|s_n'(0)|$ is uniformly bounded on $n$ by (\ref{Nehari4}) and using the Claim 2. 
\end{proof}
\begin{theorem}
Let $\lambda,h$ satisfy (\ref{Nehari1}). Then there exists a nonnegative function $ u_0 \in \mathcal{M}^+$ such that $J(u_0)= \inf_{u \in \mathcal{M}\setminus \{0\}} J(u)$. Moreover, $u_0$ is a local minimum for $J$ in $W^{m,2}_{\mathcal{N}}(\Om)$.      \label{1st solution}
\end{theorem}
\begin{proof}
Let $\{u_n\}$ be a sequence which minimizes $J$ on $\mathcal{M} \setminus \{0\}$ as in (\ref{Ekeland}). \\
\noindent\textbf{Step 1:} $\liminf_{n \to \infty} \int_{\Om} hu_n >0$ and hence $u_n \in \mathcal{M}^+$. Indeed $u_n \in \mathcal{M}$ and making some suitable adjustments
 \begin{align}
J(u_n) &= \frac{p}{2(p+2)} \int_{\Om} |\nabla^m u|^2 + \int_{\Om} \left( \frac{1}{p+2} f(u_n)u_n - F(u_n)\right)   \notag\\
        & \quad  -\lambda \frac{p+1}{p+2} \int_{\Om} hu_n < -C \lambda^{2p+8}.    \label{liminf}
\end{align}
Thanks to Lemma \ref{upper-lower} there exists $C>0$.  Now we note that $F(t) < \frac{1}{p+2}f(t)t$ for all $ t \in \mathbb{R}$. Therefore we've from (\ref{liminf}), to make the inequality consistent with sign that \[ \liminf _{n \to \infty} \int_{\Om}hu_n >0.\]

\noindent\textbf{Step 2:} $\limsup _{n \to \infty} \|u\|_{W^{m,2}_{\mathcal{N}}(\Om)} < \infty$. \\
\textit{Case 1.} If $p>0$ then by the means of Sobolev embedding we derive boundedness of  $\{u_n\}$ in $W^{m,2}_{\mathcal{N}}(\Om)$. Using the fact from (\ref{liminf}) 
$$ \int_{\Om}|\nabla^m u|^2 \leq \lambda \int_{\Om} hu_n.$$ 
\textit{Case 2.} If $p=0$ by using the fact that $\frac{1}{2}f(t)t - F(t) \geq Ct^4$ for all $t \in \mathbb{R}$ and for some $C>0$ we deduce that $\{u_n\}$ is a bounded sequence in $L^2(\Om)$. And this gives that $\{F(u_n)\}$ is a bounded sequence in $L^1(\Om)$ using (\ref{liminf}) and hence $\{u_n\}$ is a bounded sequence in $W^{m,2}_{\mathcal{N}}(\Om)$. 

\noindent \textbf{Step 3:} Existence of $u_0 \in \mathcal{M}^+$.\\
From the previous step up to a subsequence, $u_n \rightharpoonup u_0$ in $W^{m,2}_{\mathcal{N}}(\Om)$. Now from the Proposition 2.2 we note that $\{f(u_n)u_n\}$ is a bounded sequence in $L^1(\Om)$. Therefore from Vitali's convergence theorem (for details see Lemma 8.3 in \cite{Prashanth}), we get that $$\int_{\Om}f(u_n) \phi \to \int_{\Om}f(u_0)\phi, \text{ for all } \phi \in W^{m,2}_{\mathcal{N}}(\Om).$$ Hence $u_0$ will solve $(P)$, in particular $u_0 \in \mathcal{M}$. Here we note that $u_0 \neq 0$ as $h \neq 0$ that is $u_0 \in \mathcal{M} \setminus \{0\}$. We see that $ \theta_0 \leq J(u_0)$. From (\ref{liminf}) we get by using Fatou's Lemma that $\theta_0 = \liminf_{n \to \infty} J(u_n) \geq J(u_0)$. Therefore $u_0$ minimizes $J$ on $\mathcal{M} \setminus \{0\}$. Now we have to show $u_0 \in \mathcal{M}^+$. From the existence of $s_-(u_0)$ and $s_+(u_0)$ in Lemma \ref{maxima-minima} and using the fact $J(s_+(u_0)u_0) < J(s_-(u_0)u_0)$ we get $u_0 \in \mathcal{M}^+$. \\
\noindent \textbf{Step 4:} $u_0$ is a local minimum for for $J$ in $W^{m,2}_{\mathcal{N}}(\Om)$. 
\par We see that $s_+(u_0) =1$ because $u_0 \in \mathcal{M}^+$ from Step 3. Also we have from the (\ref{maxima-minima}) we have 
\begin{equation*}
s_+(u_0)=1 <s_*(u_0)
\end{equation*}
Now by the continuity of $s_*(u_0)$, for sufficiently small $\delta>0$  
\begin{equation}
1 < s_*(u_0-w), \hspace{1.5mm}\forall w \in W^{m,2}_{\mathcal{N}}(\Om), \|w\|_{W^{m,2}_{\mathcal{N}}(\Om)} < \delta.    \label{minimizer}
\end{equation}
Now by the Lemma \ref{Implicit-Inner} for $\delta>0$ small enough if necessary, let $s: \{w \in W^{m,2}_{\mathcal{N}}(\Om): \|w\| < \delta\} \to \mathbb{R}$ such that $s(w)(u_0-w) \in \mathcal{M}$ and $s(0)=1$. 
\noindent Whenever $s(w) \to 1 $ when $\|w\| \to 0$, we can assume that 
\[ s(w) < s_*(u_0-w), \hspace{1.5mm} \forall w \in W^{m,2}_{\mathcal{N}}(\Om), \|w\| < \delta.\] Hence we get $s(w)(u_0-w) \in \mathcal{M}^+$ using the above inequality and Lemma \ref{maxima-minima}. Again by using the Lemma \ref{maxima-minima} we see,
\[J(s(u_0-w) \geq J(s(w)(u_0-w)) \geq J(u_0), \hspace{1.5mm} \forall s \in [0,s_*(u_0-w)].\]Hence from (\ref{minimizer}) we observe that $J(u_0-w) \geq J(u_0)$ for every $\|w\|_{W^{m,2}_{\mathcal{N}}(\Om)} <\delta$. This shows that $u_0$ is a local minimizer. \\
\noindent{\textbf{Step 5:}} A positive local minimum for $J$. If $u_0 \geq 0$ then we get the positivity by using the strong maximum principle. In case if $u_0 \ngeq 0$ then we consider $\tilde{u}_0 = s_+(u_0)|u_0|>0 \in \mathcal{M}^+$ and also from the definition $\rho_{u_0}(s)=\rho_{|u_0|}(s)$ for all $s>0$. Therefore we get $s_*(|u_0|)=s_*(u_0)$ and from the definition of $s_+$ we deduce $s_+(u_0) \leq s_+(|u_0|)$. Hence from Step 4, $s_+(|u_0|) \geq 1$. Therefore by Lemma \ref{maxima-minima} we get $J(\tilde{u}_0) \leq J(|u_0|)$. Now using the assumption $h \geq 0$ in $\Omega$, we have $J(|u_0|) \leq J(u_0)$ and which implies that $\tilde{u}_0$ minimizes $J$ on $\mathcal{M}\setminus \{0\}$. Hence by repeating the same argument as in Step 4 we get the desired result.
\end{proof}

\section{Existence of The Second Solution}
The existence of the second solution for $(P)$ depends on whether we can apply some version of Mountain Pass Lemma. We wish to look for a solution of the form $u_1 = v + u_0$ where $u_0$ is the local minimum for the functional (\ref{energy}). Then we see that $u_1$ will solve $(P)$ whenever $v$ solves the following equation:
 $$(P_1)\hspace{1cm}\left\{
\begin{array}{lllll}\left.
\begin{array}{lllll}
\D v &= f(v+u_0)-f(u_0)\\
\hspace{8mm} v &> 0 \end{array}\right\} \;\; \text{in} \;\Om, \\
v = \De v = 0=..=\De^{m-1}u  \hspace{15.5mm}\text{ on } \partial \Omega. \end{array}\right.
$$
We can write the above PDE as following
  $$(\tilde{P}) \hspace{1cm}\left\{
\begin{array}{lllll}\left.
\begin{array}{lllll}
\D2 v &= \tilde{f}(x,v)\\
\hspace{6mm} v &> 0 \end{array}\right\} \;\; \text{in} \;\Om, \\
v = \De v = 0=..=\De^{m-1}v \text{  on } \partial \Omega, \end{array}\right.
$$
by introducing the function $\tilde{f}: \Om \times \mathbb{R} \to \mathbb{R}$ and we define by 
\begin{align*}
\tilde{f}(x,s) &= f(s+u_0(x))-f(u_0(x)) \text{ if } s \geq 0,\\
               &= 0 \hspace{2.5cm}\text{   otherwise.}
  \end{align*}
The energy functional corresponding to $(\tilde{P})$ is $J_{u_0}: {W^{m,2}_{\mathcal{N}}(\Om)} \to \mathbb{R}$ defined by 
\[ J_ {u_0}(v) = \frac{1}{2}\int_{\Om}|\nabla^m v|^2 - \int_{\Om} \tilde{F}(x,v)dx,\] where $\tilde{F}(x,s) = \int_0^s \tilde{f}(x,t)dt$.
Now onwards, we denote $J_{u_0}$ by $J_0$. These type of functionals were studied by \cite{Zhao-Chang}, \cite{Peral}. We now state the Generalized Mountain Pass Lemma that was introduced by Ghoussoub-Preiss \cite{Preiss}. 
\begin{Def}
Let $H$ be a closed subspace of the Banach Space $W^{m,2}_{\mathcal{N}}(\Om)$. We say that a sequence $\{v_n\} \subset W^{m,2}_{\mathcal{N}}(\Om)$ is a Palais-Smale sequence for $J_0$ at the level $c$ around $H$ if:
\begin{itemize}
\item[(i)] $\lim_{n \to \infty} dist(v_n,H) =0$
\item[(ii)] $\lim_{n \to \infty} J_0(v_n) = c$
\item[(iii)] $\lim_{n \to \infty} \|J_0'(v_n)\|_{(W^{m,2}_{\mathcal{N}}(\Om))^{-1}} = 0.$
\end{itemize}
And we say such a sequence a $(PS)_{H,c}$ sequence.
\end{Def}
\begin{Rem}
In case $H =W^{m,2}_{\mathcal{N}}(\Om)$, the above definition coincides with the usual Palais-Smale sequence at the level $c$.
\end{Rem}
\begin{lem} \label{mountain}
Let $H \subset W^{m,2}_{\mathcal{N}}(\Om)$ be a closed set,$c \in \mathbb{R}$. Assume $\{v_n\} \subset W^{m,2}_{\mathcal{N}}(\Om)$ be a $(PS)_{H,c}$ sequence. Then (upto a subsequence), $v_n \rightharpoonup v_0$ in $W^{m,2}_{\mathcal{N}}(\Om)$, and 
\begin{equation}
\lim _{n \to \infty} \int_{\Om} \tilde{f}(x,v_n) = \int_{\Om} \tilde{f}(x,v_0), \hspace{1cm} \lim_{n \to \infty} \int_{\Om} \tilde{F}(x,v_n) = \int_{\Om} \tilde{F}(x,v_0).     \label{Vitali}
\end{equation}      
\end{lem} 
\begin{proof}
From the fact that $\{v_n\}$ is a $(PS)_{H,c}$ sequence we have:
\begin{align}
\frac{1}{2}\int_{\Om} |\nabla^m v_n|^2 - \int_{\Om} \tilde{F}(x,v_n) &= c_0 + o_n(1),             \label{GPS1}\\
\left|\int_{\Om} \nabla^m v_n \cdot \nabla^m \phi - \int_{\Om}\tilde{f}(x,v_n) \phi \right| &\leq o_n(1) \|\phi\|_{W^{m,2}_{\mathcal{N}}(\Om)}, \hspace{5mm} \forall \phi \in W^{m,2}_{\mathcal{N}}(\Om).    \label{GPS2}
\end{align}
Now we claim that,\\
\textbf{Claim:} $\sup_n \|v_n\|_{W^{m,2}_{\mathcal{N}}(\Om)}  < \infty, \hspace{3mm}\sup_n \int_{\Om} \tilde{f}(x,v_n) < \infty.$\\
Given any $\epsilon >0$ there exists $s_{\epsilon} >0$ such that 
\begin{equation}
\int_{\Om} \tilde{F}(x,s) \leq \epsilon s \tilde{f}(x,s) \text{ for all } |s| \geq s_{\epsilon}.     \label{GPS3}
\end{equation}
Using (\ref{GPS1}) and (\ref{GPS3}), we see
\begin{align}
\frac{1}{2}\int_{\Om} |\nabla^m v_n|^2  &\leq \int_{\Om \cap \{|v_n| \leq s_{\epsilon}\}} \tilde{F}(x,v_n) + \int_{\Om \cap \{|v_n| \geq s_{\epsilon}\}} \tilde{F}(x,v_n) +c+o_n(1)            \notag \\
                 &\leq \int_{\Om \cap \{|v_n| \leq s_{\epsilon}\}}\tilde{F}(x,v_n) + \epsilon \int_{\Om} \tilde{f}(x,v_n)v_n +c+o_n(1)   \notag\\
                  &\leq C_{\epsilon} + \epsilon \int_{\Om} \tilde{f}(x,v_n)v_n.      \label{GPS4}
\end{align}
Now from (\ref{GPS4}) we obtain,
\begin{align*}
\int_{\Om}\tilde{f}(x,v_n)v_n &\leq \int_{\Om}|\nabla^m v_n|^2 + o_n(1) \|v_n\|_{W^{m,2}_{\mathcal{N}}(\Om)}\\
                             &\leq 2 C_{\epsilon} + 2\epsilon \int_{\Om} \tilde{f}(x,v_n)v_n +o_n(1)\|v_n\|_{W^{m,2}_{\mathcal{N}}(\Om)}
\end{align*}
by substituting $\phi = v_n$ in (\ref{GPS2}).\\
Hence by choosing $\epsilon$ small enough if needed we get 
\begin{equation}
 \int_{\Om} \tilde{f}(x,v_n)v_n \leq \frac{2C_{\epsilon}}{1-2\epsilon} + o_n(1)\|v_n\|_{W^{m,2}_{\mathcal{N}}(\Om)}.   \label{GPS5}
\end{equation}
We conclude the claim using (\ref{GPS5}), (\ref{GPS2}) and also $\sup_n\int_{\Omega}\tilde{f}(x,v_n)v_n <\infty$. \\
Since $\{v_n\} \subset W^{m,2}_{\mathcal{N}}(\Om)$ is bounded, up to a subsequence, $v_n \rightharpoonup v_0$ in $W^{m,2}_{\mathcal{N}}(\Om)$, for some $v_0 \in W^{m,2}_{\mathcal{N}}(\Om)$.\\
To prove (\ref{Vitali}) we consider $A$ to be a $2m$ dimensional Lebesgue measure of a set $A \subset \mathbb{R}^{2m}$.
\par Let $C=\sup_n \int_{\Omega}|\tilde{f}(x,v_n)v_n|<\infty$ from the above claim. Given $\epsilon >0$, we define 
\[\mu_{\epsilon}= \max_{x \in \bar{\Omega}, |s|\leq \frac{2C}{\epsilon}}|\tilde{f}(x,s)s|.\] Then, for any $A \subset \Omega$ with $|A| \leq \frac{\epsilon}{2C}$, we have
\begin{align*}
 \int_{A} |\tilde{f}(x,v_n)| &\leq \int_{A \cap \{|v_n| \geq \frac{2C}{\epsilon}\}} \frac{|\tilde{f}(x,v_n)v_n|}{|v_n|} + \int_{A \cap \{|v_n| \leq \frac{2C}{\epsilon}\}} |\tilde{f}(x,v_n)|\\
 &\leq \frac{\epsilon}{2} + \mu_{\epsilon}|A| \leq \epsilon.
\end{align*}
Hence $\{\tilde{f}(x,v_n)\}$ is an equi-integrable family in $L^1(\Omega)$ and so is $\{\tilde{F}(x,v_n)\}$(we note that $|\tilde{F}(x,t)| \leq C_1 |\tilde{f}(x,t)|$ for all $x \in \bar{\Omega}, t \in \mathbb{R}$, for some $C_1>0$). By applying the Vitali's convergence theorem we get conclude the lemma. 
 \end{proof}

Certainly $J_0(0)=0$ and $v=0$ is a local minimum for $J_0$. Also we have $$ \lim_{s \to \infty}J_0(sv) = -\infty \text{ for any } v \in W^{m,2}_{\mathcal{N}}(\Om)\setminus \{0\}.$$  Hence we can fix $e \in W^{m,2}_{\mathcal{N}}(\Om)\setminus \{0\}$ such that $J_0(e) <0$. Now we define the mountain pass level $$ c_0 = \inf_{\gamma \in \Gamma} \sup_{s \in [0,1]} J_0(\gamma(s)).$$
Where $\Gamma = \{ \gamma \in C([0,1],W^{m,2}_{\mathcal{N}}(\Om)): \gamma(0)=0, \gamma(1) = e\}$. Then from the definition of $c_0$ it follows $c_0 \geq 0$. Define $R_0 =\|e\|_{W^{m,2}_{\mathcal{N}}(\Om)},$ we note that $\inf\{J_0(v):\|v\|_{W^{m,2}_{\mathcal{N}}(\Om)}=R\} = 0$ for all $R \in (0,R_0)$. And we now let $H = W^{m,2}_{\mathcal{N}}(\Om)$ if $c_0 >0$ and $H=\{\|v\|_{W^{m,2}_{\mathcal{N}}(\Om)} =\frac{R_0}{2}\}$ if $c_0=0$. We now state as now the lemma giving upper bound for $c_0$
\begin{lem}The upper bound of the Mountain Pass level is below
\begin{equation}
c_0 < \frac{(4 \pi)^m m!}{2}= \frac{\beta_{2m,m}}{2}. \label{MPL1}
\end{equation}
\end{lem}
\begin{proof}
Without loss of generality we can assume that the unit ball $B_0(1) \subset \Om$. For any $\epsilon >0$ we define  
\begin{align}
{\tilde{\tau}}_n(x) := \left\{
\begin{array}{lllll}
\sqrt{\frac{1}{2M}\log n}+\frac{1}{\sqrt{2M \log n}} \sum_{\gamma =1}^{m-1} \frac{(1-k|x|^2)^{\gamma}}{\gamma} \hspace{6mm} |x| \in [0, \frac{1}{\sqrt{n}})\\
-\sqrt{\frac{2}{M \log n}} \log |x|,                \hspace{3.75cm} |x| \in [\frac{1}{\sqrt{n}},1),\\
\chi_n(x),       \hspace{5.59cm} |x| \in [1,\infty)
\end{array}\right.
\end{align}
where \[ M = \frac{(4 \pi)^m (m-1)!}{2}, \hspace{1cm} \chi_n \in \mathcal{C}^{\infty}_0(\Om), \quad \chi_n|_{\partial B_1(0)}= \chi_n|_{\partial \Omega} =0.\]
Furthermore, for $\gamma=1,2,..,m-1, D^{\gamma} \chi_n|_{\partial B_1(0)}= (-1)^{\gamma}(\gamma-1)! \sqrt{\frac{2}{M \log n}}$,  $\Delta^j \chi_n|_{\partial \Omega} = 0$ for $j =0,1,2,...,[(m-1)/2]$ 
and $ \chi_n, |\nabla \chi_n|, \De \chi_n$ are all $O\left(\frac{1}{\sqrt{2 \log n}}\right)$. Then, $\tilde{\tau}_n \in W^{m,2}_{\mathcal{N}}(\Om).$
Now we normalize $\tilde{\tau}_n$, setting $$ {\tau}_n := \frac{\tilde{\tau}_n}{\|\tilde{\tau}_n\|_{W^{m,2}_{\mathcal{N}}(\Om)}} \in  W^{m,2}_{\mathcal{N}}(\Om).$$
Suppose (\ref{MPL1}) is not true. This means that, for some $s_n>0$ (see \cite{Lam-Lu}),
\[J_0(s_n \tau_n)= \sup_{s>0}J_0(s \tau_n) \geq \frac{(4 \pi)^mm!}{2}  \hspace{1cm} \forall n.\]
Hence 
\begin{equation}
 \frac{s_n^2}{2} - \int_{\Om} \tilde{F}(x,s_n \tau_n) \frac{(4 \pi)^m m!}{2} \hspace{1cm} \forall n.     \label{MPL2}
\end{equation}
 It follows that $\frac{d}{ds}J_0(s \tau_n)=0$ at the point of maximum $s=s_n$ for $J_0$, we get 
\begin{equation}
s_n^2 = \int_{\Om} \tilde{f}(x,s_n \tau_n)(s_n \tau_n).     \label{MPL3}
\end{equation} 

 Now we note that from the definition of $\tilde{f}$ we see that $\inf_{x \in \bar{\Om}} \tilde{f}(x,s) \geq e^{s^2}$ for $|s|$ large. Then from (\ref{MPL2}) we get for sufficiently large $n$  
 \begin{align}
 s_n^2 &\geq \int_{\{|x| \leq \frac{1}{\sqrt{n}}\}} \tilde{f}(x,s_n \tau_n)(s_n \tau_n) \geq \int_{\{|x| \leq \frac{1}{\sqrt{n}}\}} e^{s_n^2 \tau_n ^2}(s_n \tau_n)        \notag\\
        &\geq e^{{s_n^2} \frac{\log n}{2M}} \frac{s_n}{\sqrt{2M}} \sqrt{\log n} \frac{\alpha_{2m}}{n^m}   \notag\\
        &=\frac{\alpha_{2m}}{\sqrt{2M}} e^{\left(\frac{s_n^2}{2M}-m\right)\log n} s_n (\log n)^{\frac{1}{2}},    \label{MPL4}
 \end{align}
  where $\alpha_{2m}$ is the volume of the unit ball in $\mathbb{R}^{2m}$. Using the fact $s_n^2 \geq (4 \pi)^m m!$ from (\ref{MPL2}) and (\ref{MPL4}) it follows that $s_n$ is bounded and also $s_n^2 \to (4 \pi)^m m!$. Also from (\ref{MPL4}) we note \[ s_n \geq \frac{w_{2m}}{\sqrt{2M}} (\log n)^{\frac{1}{2}}, \text{ for all large } n\] which gives the contradiction. 
\end{proof}
We now prove the theorem regarding the existence of second solution. 
\begin{theorem}  \label{2nd solution}
Given a local minimum $u_0$ of $J$ in $W^{m,2}_{\mathcal{N}}(\Om)$, there exists a point $v_0 \in W^{m,2}_{\mathcal{N}}(\Om)$ with $v_0 >0$ in $\Om$, such that $J_0'(v_0)=0$. 
\end{theorem}
\begin{proof}
From Lemma \ref{MPL1} we have $c_0 \in \left[0, \frac{(4 \pi)^ m m!}{2}\right).$ Consider $\{v_n\}$ be a Palais-Smale sequence for $J_0$ at the level $c_0$ around $H$ (such a $(PS)_{H,c_0}$ sequence exists \cite{Preiss}). Then up to a subsequence $v_n \rightharpoonup v_0$ in $W^{m,2}_{\mathcal{N}}(\Om)$ for some $v_0 \in W^{m,2}_{\mathcal{N}}(\Om)$ by Lemma (\ref{mountain}) and (\ref{Vitali}) holds. We can easily check that $v_0$ is a solution of $(\tilde{P})$ and therefore a critical point of $J_0$. It remains to show that $v_0$ is not a trivial solution. \\
\textbf{Case I.} $c_0 = 0,  v_0 = 0$. We note that $H= \{\|v\|_{W^{m,2}_{\mathcal{N}}(\Om)} = \frac{R_0}{2}\}$ in this case. As $\{v_n\}$ is a $(PS)_{H,c}$ sequence we have $v_n \to 0$ strongly in $W^{m,2}_{\mathcal{N}}(\Om)$. From the fact that $dist(v_n,H) = 0$ and $H$ is closed we conclude that $v_n \in H$ and which implies that $v_0 \in H$ and $v_0$ is different from $0$.\\
\textbf{Case II.} $c_0 \in \left(0,\frac{(4 \pi)^m m!}{2}\right), v_0 =0$. Using the fact that $J_0(v_n) \to c_0$ we see that for given any $\epsilon > 0, \|v_n\|^2_{W^{m,2}_{\mathcal{N}}(\Om)} \leq (4 \pi)^m m! -\epsilon$ for all large $n$.  Let $0 < \delta < \frac{\epsilon}{(4 \pi)^m m!}$ and $q = \frac{(4 \pi)^m m!}{(1+\delta)((4 \pi)^m m!-\epsilon)} >1$. We have  
$$\int_{\Om} |\tilde{f}(x,v_n)v_n|^q \leq C \int_{\Om} e^{((1+\delta)q\|v_n\|^2)\left(\frac{v_n^2}{\|v_n\|^2}\right)^2},$$
since $\sup_{x \in \bar{\Om}}|\tilde{f}(x,s)s| \leq C e^{(1+\delta)s^2}$, for all $ s \in \mathbb{R}$, for some $C>0$. Now from the Tarsi's embedding \ref{Tarsi} we get that $\sup_{x \in \bar{\Om}} \int_{\Om} |\tilde{f}(x,v_n)v_n|^q  < \infty$ since $(1+\delta)q\|v_n\|^2 \leq (4 \pi)^m m!$. Also by Vitali's convergence theorem we get  $\int_{\Om} \tilde{f}(x,v_n)v_n \to 0$ as $n \to \infty$ since $v_n \to 0$ pointwise almost everywhere in $\Om$. Which implies 
\begin{align*}
 o_n(1) \|v_n\|_{W^{m,2}_{\mathcal{N}}(\Om)} &= \langle J_0'(v_n),v_n \rangle = \frac{1}{2} \int_{\Om}|\nabla v_n|^2 - \int_{Om} \tilde{f}(x,v_n)v_n\\
 &= \frac{1}{2} \int_{\Om} |\nabla^m v_n|^2+ o_n(1)
 \end{align*}
which contradicts the fact $\frac{1}{2} \int_{\Om} |\nabla^m v_n|^2 \to c_0$ as $n \to \infty$. Therefore $v_0$ is not identically $ 0$ in $\Om$. And positivity of $v_0$ comes from the fact that $\tilde{f}(x,s) \geq 0$ for all $(x,s) \in \Om \times \mathbb{R}$ and using the maximum principle.
\end{proof}
\section{Proof of Theorem \ref{Multiplicity}}
Define $\lambda_* = \mu C_0^{\frac{p+3}{p+4}}|\Om|^{-\frac{p+2}{2p+8}}$ where $C_0$ is same as in the Proposition (\ref{lambdastar}). Then condition (\ref{Nehari1}) is true whenever $0<\lambda < \lambda_*$. From the Theorem \ref{1st solution} and \ref{2nd solution} we show the existence of at least two positive solutions for $(P)$. \\
Let $\phi_1$ be the eigen function of $(-\Delta)^m$ on $W^{m,2}_{\mathcal{N}}(\Om)$. Define
$$\lambda^* = p \left( \frac{\lambda_1}{p+1}\right)^{\frac{p+1}{p}} \left( \frac{\int_{\Om}\phi_1}{\int_{\Om}h \phi_1}\right).$$
 We prove that there is no solution of $(P)$ when $\lambda > \lambda^*$. Assume that $u_{\lambda}$ be a solution of $(P)$. By multiplying $\phi_1$ with $(P)$ and performing integration by parts over $\Omega$, we get
\begin{align*}
\int_{\Om} \D u_{\lambda} \phi_1 = \int_{\Om}f(u_{\lambda})\phi_1 + \lambda \int_{\Om}h \phi_1
\end{align*}
implies 
\begin{equation}
\lambda \int_{\Om} h \phi_1 = \int_{\Om}(\lambda_1u_{\lambda}-f(u_{\lambda}))\phi_1                 \label{lambdaupper}
\end{equation}
We see that $\lambda_1t-f(t) \leq \lambda_1-\mu t^{p+1}=\Theta(t)$ for all $t >0$. The global maximum for the function $\Theta$ is $p \left(\frac{\lambda_1}{p+1} \right)^{\frac{p+1}{p}}$ on $(0,\infty)$. 
Then from (\ref{lambdaupper}) and the definition of $\lambda^*$ we get $\lambda \leq \lambda^*$. This completes Theorem \ref{Multiplicity}.
\section*{Acknowledgement} The author would like to thank Prof. S Prashanth for several helpful discussions.

\end{document}